\documentclass[11pt]{article}
\usepackage{amsmath}
\usepackage{amssymb}
\usepackage{amsthm}
\usepackage[usenames]{color}
\usepackage{amscd}
\usepackage{dsfont}
\usepackage{indentfirst}

\usepackage[colorlinks=true,linkcolor=blue,filecolor=red,
citecolor=webgreen]{hyperref}
\definecolor{webgreen}{rgb}{0,.5,0}

\usepackage[np]{numprint}

\numberwithin{equation}{section}

\def\C{{\mathds{C}}}

\def\N{{\mathds{N}}}

\def\1{{\bf 1}}

\newtheorem{theorem}{Theorem}

\newtheorem{lemma}{Lemma}
\newtheorem{cor}{Corollary}

\newtheorem{prop}{Proposition}
\newtheorem{remark}{Remark}

\begin{document}

\title{{\bf Additive arithmetic functions meet the inclusion-exclusion principle: Asymptotic formulas concerning the GCD and LCM of several integers}}
\author{Olivier Bordell\`es \\  2 All\`ee de la Combe \\ 43000 Aiguilhe, France \\
E-mail: {\tt borde43@wanadoo.fr} \\
and \\  
L\'aszl\'o T\'oth\thanks{The research was financed by NKFIH in Hungary, within the framework of the 2020-4.1.1-TKP2020 3rd thematic programme 
of the University of P\'ecs.} \\ Department of Mathematics \\
University of P\'ecs \\
Ifj\'us\'ag \'utja 6, 7624 P\'ecs, Hungary \\
E-mail: {\tt ltoth@gamma.ttk.pte.hu}}
\date{}
\maketitle

\centerline{Lithuanian Math. J. {\bf 62} (2022), no. 2, 150--169}

\begin{abstract} We obtain asymptotic formulas for the sums $\sum_{n_1,\ldots,n_k\le x} f((n_1,\ldots,n_k))$ and \\ $
\sum_{n_1,\ldots,n_k\le x} f([n_1,\ldots,n_k])$, involving the gcd and lcm of the integers $n_1,\ldots,n_k$, where 
$f$ belongs to certain classes of additive arithmetic functions. In particular, we consider the generalized omega function 
$\Omega_{\ell}(n)= \sum_{p^\nu \mid\mid n} \nu^{\ell}$ investigated by Duncan (1962) and Hassani (2018), 
and the functions $A(n)=\sum_{p^\nu \mid\mid n} \nu p$, $A^*(n)= \sum_{p \mid n} p$, $B(n)=A(n)-A^*(n)$ studied by
Alladi and Erd\H{o}s (1977).  As a key auxiliary result we use an inclusion-exclusion-type identity. 
\end{abstract}




\section{Motivation} \label{Sect_Motivation}

Throughout the paper we use the notation: $\N=\{1,2,\ldots\}$, $\N_0=\{0,1,2,\ldots \}$, $(n_1,\ldots,n_k)$ and $[n_1,\ldots,n_k]$ 
denote the greatest common divisor (gcd) and least common multiple (lcm) of $n_1,\ldots, n_k\in \N$, $\1(n)=1$ ($n\in \N$), $\mu$ is the M\"obius function, 
$\tau(n)=\sum_{d\mid n} 1$,  $\omega(n)=\sum_{p\mid n} 1$, $\Omega(n)=\sum_{p^\nu \mid\mid n} 1$, $A(n)=\sum_{p^\nu \mid\mid n} \nu p$, 
$A^*(n)= \sum_{p \mid n} p$, $B(n)=A(n)-A^*(n)$, $f*g$ denotes the convolution of the arithmetic functions $f$ and $g$, $\gamma$ is Euler's constant, 
$\langle {n \atop k} \rangle $ are the (classical) Eulerian numbers. 

Let $f:\N \to \C$ be an arithmetic function and let $k\in \N$. We are interested in asymptotic formulas for the sums
\begin{equation*}
G_{f,k}(x):= \sum_{n_1,\ldots,n_k\le x} f((n_1,\ldots,n_k)) 
\end{equation*}
and
\begin{equation*}
L_{f,k}(x): = \sum_{n_1,\ldots,n_k\le x} f([n_1,\ldots,n_k]).
\end{equation*}

By using the general identity 
\begin{equation*}
G_{f,k}(x) =  \sum_{d\le x} (\mu*f)(d) \lfloor x/d \rfloor^k,
\end{equation*}
valid for every function $f$, see Lemma \ref{Lemma_gcd}, it is possible to deduce asymptotic formulas for $G_{f,k}(x)$ in the case of various functions $f$. 
For example, if $k\ge 3$, then
\begin{equation} \label{gcd_recipr}
\sum_{n_1,\ldots,n_k\le x} \frac1{(n_1,\ldots,n_k)}=  \frac{\zeta(k+1)}{\zeta(k)} x^k+ O(x^{k-1}), 
\end{equation}
\begin{equation} \label{tau_gcd}
\sum_{n_1,\ldots,n_k\le x} \tau((n_1,\ldots,n_k)) =  \zeta(k) x^k+ O(x^{k-1}), 
\end{equation}
where $\tau(n)$ is the divisor function. See \cite[Sect.\ 1]{HLT2020} and \cite[Th.\ 3.6]{TotZha2018}.

It is more difficult to obtain asymptotic formulas for the sums $L_{f,k}(x)$ concerning the lcm of integers.
If the function $f$ is multiplicative, then $f([n_1,\ldots,n_k])$, and also $f((n_1,\ldots,n_k))$, are multiplicative functions of $k$ variables. 
Therefore, the multiple Dirichlet series 
\begin{equation*}
\sum_{n_1,\ldots,n_k=1}^{\infty} \frac{f([n_1,\ldots,n_k])}{n_1^{s_1}\cdots n_k^{s_k}} 
\end{equation*}
can be expanded into an Euler product, and the multiple convolution method can be used to deduce asymptotic formulas. For example, 
the counterpart of  \eqref{tau_gcd} is
\begin{equation*}
\sum_{n_1,\ldots,n_k\le x} \tau([n_1,\ldots,n_k]) =  x^k Q_k(\log x)+ O(x^{k-1+\theta+\varepsilon}), 
\end{equation*}
where $k\ge 2$, $Q_k(t)$ is a polynomial in $t$ of degree $k$ and $\theta$ is the exponent in the Dirichlet divisor problem. See \cite[Th.\ 3.4]{TotZha2018}.
This approach does not furnish a formula with remainder term, as a counterpart of \eqref{gcd_recipr}. It was only proved in \cite[Th.\ 2.3]{HLT2020} that
for $k\ge 3$,
\begin{equation*} 
S_k(x):= \sum_{n_1,\ldots,n_k\le x} \frac1{[n_1,\ldots,n_k]} \asymp (\log x)^{2^k-1},
\end{equation*}
and conjectured that 
\begin{equation*} 
S_k(x)=P_{2^k-1}(\log x) + O(x^{-r}),
\end{equation*}
where $P_{2^k-1}(t)$ is a polynomial in $t$ of degree $2^k-1$ and $r$ is a positive real number. This conjecture 
was proved in \cite{EST2021} by a different method, using analytic techniques. See \cite{EST2021,HLT2020,Tot2014,TotZha2018} for more details.

In this paper we obtain asymptotic formulas for the sums $G_{f,k}(x)$ and $L_{f,k}(x)$ in the case of certain classes 
of additive functions. In particular, we consider the generalized omega function 
$\Omega_{\ell}(n)= \sum_{p^\nu \mid\mid n} \nu^{\ell}$ investigated by Duncan \cite{Dun1962} and Hassani \cite{Has2018}, 
and the functions $A(n)=\sum_{p^\nu \mid\mid n} \nu p$, $A^*(n)= \sum_{p\mid n} p$, $B(n)=A(n)-A^*(n)$ studied by
Alladi and Erd\H{o}s \cite{AllErd1977}.  A key identity of our approach is the application of the inclusion-exclusion principle to additive functions. 
See Proposition \ref{Prop_key}, which may be known in the literature, but we could not find any reference. Other key results used in the proofs are Saffari's estimate obtained 
for the sum  $\sum_{n\le x} \omega(n)$ and the estimate for $\sum_{n\le x} A(n)$, where $A(n)$ is the Alladi-Erd\H{o}s function. See \cite{AllErd1977,Saf1970,Ten2015}. The main results on the asymptotic formulas are formulated in Section \ref{Sect_Main}. 
Some preliminary lemmas needed to the proofs are included in Section \ref{Sect_Prelim}, and the proofs of the main results are given in Section \ref{Sect_Proofs}.

\section{Additive functions and the inclusion-exclusion principle}

We recall that an arithmetic function $f:\N \to \C$ is additive if $f(mn)=f(m)+f(n)$ for all $m,n\in \N$ with $(m,n)=1$. If $f$ is additive, then $f(n)=\sum_{p^\nu \mid\mid n} f(p^{\nu})$. Some examples of additive functions are $\log n$, $\omega(n)$ and $\Omega(n)$. 

If $f$ is additive, then $f((m,n))+f([m,n]) = f(m)+f(n)$ holds for every $m,n\in \N$.  To see this, it is enough to consider the case when $m$ and $n$ are powers of the same prime $p$, 
namely $m=p^a$, $n=p^b$, where $a,b\in \N_0$. Now,  $f(p^{\max(a,b)})+ f(p^{\min(a,b)}) = f(p^a)+f(p^b)$ trivially holds. In a similar way, if $f$ is additive, then for every 
$n_1,n_2, n_3\in \N$,
\begin{equation*}
f([n_1,n_2,n_3])=f(n_1)+f(n_2)+f(n_3)-f((n_1,n_2))-f((n_1,n_3))-f((n_2,n_3))
\end{equation*}
\begin{equation*}
+f((n_1,n_2,n_3)).
\end{equation*}

We generalize these identities to several integers. 

\begin{prop} \label{Prop_key} Let $f:\N \to \C$ be an additive function, let $k\in \N$ and $n_1,\ldots,n_k\in \N$. Then
\begin{equation} \label{id_additive_gen}
f([n_1,\ldots,n_k]) =\sum_{1\le j\le k} (-1)^{j-1} \sum_{1\le i_1< \cdots < i_j \le  k} f((n_{i_1},\ldots,n_{i_j})).
\end{equation}
\end{prop}

\begin{proof}  It is enough to prove identity \eqref{id_additive_gen} if $n_1=p^{\nu_1}, \ldots, n_k=p^{\nu_k}$ are powers of the same prime $p$ 
with $\nu_1,\ldots,n_k\in \N_0$, that is,
\begin{equation} \label{id_p}
f(p^{\max(\nu_1,\ldots, \nu_k)}) = \sum_{1\le j\le k} (-1)^{j-1} \sum_{1\le i_1< \cdots < i_j \le  k} f(p^{\min(\nu_{i_1},\ldots,\nu_{i_j})}).
\end{equation}

By symmetry we can assume that $\nu_1\le \nu_2 \le \cdots \le \nu_k$. Then the LHS of \eqref{id_p} is $f(p^{\nu_k})$. Let $1\le \ell \le k$. On the RHS of \eqref{id_p}, 
for a fixed $j$, the term $f(p^{\nu_{\ell}})$ appears if $i_1=\ell < i_2<\cdots < i_j\le k$. This happens $\binom{k-\ell}{j-1}$ times.  Hence on the RHS the coefficient of 
$f(p^{\nu_{\ell}})$ is
\begin{equation*}
\sum_{1\le j\le k} (-1)^{j-1}  \binom{k-\ell}{j-1} = \sum_{1\le j\le k-\ell+1} (-1)^{j-1}  \binom{k-\ell}{j-1} = 
\begin{cases} 1, & \text{ if $\ell =k$,} \\ 0 , & \text{ if $\ell \ne k$,} \end{cases}
\end{equation*}
which completes the proof, similar as in the proof of the inclusion-exclusion principle.
\end{proof}

\begin{remark} {\rm Consider the additive function $\omega(n)$. For given $n_1,\ldots,n_k\in \N$ let $A_j$ be the set of prime factors of $n_j$ ($1\le j\le 
k$). 
Then $\omega(n_j)= \# A_j$,
$ \omega((n_j,n_{\ell}))= \# (A_j \cap A_{\ell})$, etc., $\omega([n_1,\ldots,n_k]) = \# (A_1\cup \ldots \cup A_k)$, and  \eqref{id_additive_gen} reduces to 
the classical 
inclusion-exclusion principle for the sets $A_j$.
}
\end{remark}

We also recall that a function $g:\N \to \C$ is multiplicative if $g(mn)=g(m)g(n)$ for all $m,n\in \N$ with $(m,n)=1$. If $g$ is multiplicative, then
$g(n)=\prod_{p^\nu \mid\mid n} g(p^{\nu})$. If $f$ is additive, then the function $g(n)=2^{f(n)}$ is multiplicative. Conversely, if $g$ is multiplicative (and positive), then the function 
$f(n)=\log g(n)$ is additive. If $g$ is multiplicative, then in a similar manner as above, 
\begin{equation} \label{g_mult}
g((m,n))g([m,n]) = g(m)g(n)
\end{equation}
holds for every $m,n\in \N$. This is well-known and is included in many textbooks. See., e.g., \cite[Ex.\ 1.9]{McC1986}. Also see \cite{Hau2012} for the 
related notion of semimultiplicative (Selberg multiplicative) functions, and \cite{ApoZuc1964} for some other similar two variables identities. 

More generally than \eqref{g_mult}, we have the next  result.

\begin{cor} Let $g:\N \to \C$ be a nonvanishing multiplicative function, let $k\in \N$ and $n_1,\ldots,n_k\in \N$. Then
\begin{equation*}
g([n_1,\ldots,n_k]) =\prod_{1\le j\le k}  \left( \prod_{1\le i_1< \cdots < i_j \le  k} g((n_{i_1},\ldots,n_{i_j}))\right)^{(-1)^{j-1}}.
\end{equation*}
\end{cor}

\begin{proof} Apply (formally) Proposition \ref{Prop_key} to the additive function $f(n)=\log g(n)$.
\end{proof}

If $g(n)=n$, then this is known (see, e.g., \cite[Ex.\, 3.4.56]{Ros2000}) and gives the lcm of several integers in terms of the gcd's:
\begin{equation*}
[n_1,\ldots,n_k] = \frac{n_1\cdots n_k}{(n_1,n_2)(n_1,n_3)\cdots (n_{k-1},n_k)} \cdot \frac{(n_1,n_2,n_3)\cdots}{(n_1,n_2,n_3,n_4)\cdots} \cdots .
\end{equation*}

\section{Asymptotic formulas for multivariable sums} \label{Sect_Main}

\subsection{The class ${\cal F}_0$ of omega-type functions} \label{Sect_F_0}

Let ${\cal F}_0$ denote the class of additive functions $f:\N \to \C$ such that $f(p)=1$ for every prime $p$ and $f(p^{\nu}) \ll \nu^{\ell}$ holds uniformly 
for the primes $p$ and $\nu \ge 2$, where $\ell \in \N_0$ is some integer. For example, the functions $\omega(n)$ and $\Omega(n)$ are in ${\cal F}_0$. More generally, the 
function $\Omega_{\ell}(n)=\sum_{p^\nu \mid\mid n} \nu^{\ell}$, with $\ell \in \N_0$, is in the class ${\cal F}_0$. Note that $\Omega_0(n)=\omega(n)$, $\Omega_1(n)=\Omega(n)$. The function $\Omega_{\ell}(n)$ was defined by Duncan \cite{Dun1962} and an asymptotic formula for $\sum_{n\le x} 
\Omega_{\ell}(n)$ was obtained in that paper.

Another example of a function in the class ${\cal F}_0$ is given by $T_{\ell}(n)= \sum_{p^\nu \mid\mid n} \left(\binom{\nu}{\ell}\right)$, 
where $\left(\binom{n}{k}\right) = \binom{n+k-1}{k}=\frac{n(n+1)\cdots (n+k-1)}{k!}$ is the number of $k$--combinations with repetitions of $n$ elements. Observe that $T_0(n)=\omega(n)$, $T_1(n) = \Omega(n)$. 

Saffari \cite{Saf1970} proved that the estimate  
\begin{equation} \label{est_Saffari}
\sum_{n\le x} \omega(n)= x\log \log x + M x + x \sum_{j=1}^N \frac{a_j}{(\log x)^j} + O\left(\frac{x}{(\log x)^{N+1}}\right)     
\end{equation}
holds for every fixed integer $N\ge 1$, where $M$ is the Mertens constant defined by
\begin{equation*}
M = \gamma + \sum_p  \left( \log \left(1 -\frac1{p}\right)  + \frac1{p} \right) \approx 0.2614
\end{equation*}
and the constants $a_j$ ($1\le j\le N$) are given by 
\begin{equation} \label{const_a_j}
a_j = - \int_1^{\infty} \frac{t-\lfloor t \rfloor}{t^2}(\log t)^{j-1} \, dt,
\end{equation}
in particular, $a_1=\gamma-1$. Also see \cite[Sect.\ 4.3.11]{Bor2020}.

By using the proximity of the functions $f\in {\cal F}_0$ and the function $\omega$ we first prove the following result.

\begin{theorem} \label{Th_f_additive} Let $f$ be a function in class ${\cal F}_0$. Then
\begin{equation*}
\sum_{n\le x} f(n) = x \log \log x + C_f x + x \sum_{j=1}^N \frac{a_j}{(\log x)^j} + O\left(\frac{x}{(\log x)^{N+1}}\right)     
\end{equation*}
for every fixed $N\ge 1$, where the constant $C_f$ is given by
\begin{equation*}
C_f = \gamma + \sum_p  \left( \log \left(1 -\frac1{p}\right)  + \left(1-\frac1{p}\right)  \sum_{\nu=1}^{\infty} \frac{f(p^{\nu})}{p^{\nu}} \right),
\end{equation*}
and the constants $a_j$  \textup{($1\le j\le N$)} are as in \eqref{const_a_j}.
\end{theorem}

Note that a weaker asymptotic formula for the sum $\sum_{n\le x} f(n)$ is given in \cite[Th.\ 6.19]{DeKLuc2012} under the more restrictive conditions $f$ additive, $f(p)=1$ and for all primes $p$ 
and $f(p^{\nu})-f(p^{\nu-1}) = O(1)$ uniformly for the primes $p$ and $\nu \ge 2$, satisfied by the functions $\omega(n)$ and $\Omega(n)$, but not by $\Omega_{\ell}(n)$ and $T_{\ell}(n)$ with $\ell \ge 2$.

\begin{cor} \label{Cor_Omega_ell} The estimate of Theorem \ref{Th_f_additive} holds for the function $f(n)=\Omega_{\ell}(n)$ \textup{($\ell \in \N_0$)} with the constant
\begin{equation*}
C_{\Omega_{\ell}} = \gamma + \sum_p  \left( \log \left(1 -\frac1{p}\right)  +  \frac1{p} \left(1-\frac1{p}\right)^{-\ell} \ \sum_{\nu=0}^{\ell -1 } 
\frac{\langle {\ell \atop \nu} \rangle }{p^{\nu}} \right),
\end{equation*}
where $\langle {\ell \atop \nu} \rangle$ are the Eulerian numbers to be defined in Section \ref{Sect_Eulerian_numbers}, and the inner sum is considered to be $1$ if $\ell =0$.
\end{cor}

The result of Corollary \ref{Cor_Omega_ell} was proved by Hassani \cite{Has2018}, also by using Saffari's estimate 
for the sum  $\sum_{n\le x} \omega(n)$, but invoking some different arguments and without referring to the Eulerian numbers.

\begin{cor} \label{Cor_T_ell} The estimate of Theorem \ref{Th_f_additive} holds for the function $f(n)=T_{\ell}(n)$ \textup{($\ell \in \N_0$)} with the constant
\begin{equation*}
C_{T_{\ell}} = \gamma + \sum_p  \left( \log \left(1 -\frac1{p}\right)  + \frac1{p} \left(1-\frac1{p}\right)^{-\ell}\right).
\end{equation*}
\end{cor}

Next we deduce the following estimates for the sums $G_{f,k}(x)$ with $k\ge 2$.

\begin{theorem} \label{Th_f_additive_gcd} Let $f$ be a function in class ${\cal F}_0$ and let $k\in \N$, $k\ge 2$. Then
\begin{equation*}
\sum_{n_1,\ldots,n_k\le x} f((n_1,\ldots,n_k)) = D_{f,k} x^k + \begin{cases} O(x^{k-1}), & \text{ if $k\ge 3$}, \\ O(x\log \log x), & \text{ if $k=2$},
\end{cases}
\end{equation*}
where the constant $D_{f,k}$ is given by
\begin{equation} \label{D_f_k}
D_{f,k} = \sum_p   \left(1-\frac1{p^k}\right)  \sum_{\nu=1}^{\infty} \frac{f(p^{\nu})}{p^{\nu k}}.
\end{equation}
\end{theorem}

\begin{cor} \label{Cor_Omega_ell_gcd} The estimate of Theorem \ref{Th_f_additive_gcd} holds for the functions $f(n)=\Omega_{\ell}(n)$ and $f(n)=T_{\ell}(n)$ \textup{($\ell \in \N_0$)} with the constants
\begin{equation*}
D_{\Omega_{\ell},k} = \sum_p \frac1{p^k} \left(1-\frac1{p^k}\right)^{-\ell} \ \sum_{\nu=0}^{\ell -1 } 
\frac{\langle {\ell \atop \nu} \rangle }{p^{\nu k}},
\end{equation*}
where the inner sum is considered to be $1$ if $\ell =0$, and 
\begin{equation*}
D_{T_{\ell},k} =  \sum_p \frac1{p^k} \left(1-\frac1{p^k}\right)^{-\ell}.
\end{equation*}

In particular, Theorem \ref{Th_f_additive_gcd} holds for the functions $\omega(n)$ and $\Omega(n)$ with the constants
\begin{equation*}
D_{\omega,k} = \sum_p  \frac1{p^k}, \quad 
D_{\Omega,k} =  \sum_p \frac1{p^k-1}.
\end{equation*}
\end{cor}

In what follows we obtain our estimates for the sums $L_{f,k}(x)$ involving the lcm of integers.

\begin{theorem} \label{Th_f_additive_lcm} Let $f$ be a function in class ${\cal F}_0$. Then for every $k\in \N$, $k\ge 2$,
\begin{equation*}
\sum_{n_1,\ldots, n_k\le x} f([n_1,\ldots,n_k]) = k x^k \log \log x + E_{f,k} x^k + kx^k  \sum_{j=1}^N \frac{a_j}{(\log x)^j} +
O\left(\frac{x^k}{(\log x)^{N+1}}\right),
\end{equation*}
for every fixed $N\ge 1$, where the constant $E_{f,k}$ is given by
\begin{equation*}
E_{f,k} = k C_f - \sum_{j=2}^k (-1)^j \binom{k}{j} D_{f,j},
\end{equation*}
and the constants $a_j$  \textup{($1\le j\le N$)} are as in \eqref{const_a_j}.
\end{theorem}

\begin{cor} The estimate of Theorem \ref{Th_f_additive_lcm} holds for the functions $\Omega_{\ell}(n)$ and $T_{\ell}(n)$ with $\ell \in \N_0$. In particular, it holds for the functions $\omega(n)$ and $\Omega(n)$ with the constants
\begin{equation*}
E_{\omega,k} = k\gamma +\sum_p \left( k\log \left(1 -\frac1{p}\right)+ 1- \left(1-\frac1{p}\right)^k \right).
\end{equation*}
and 
\begin{equation*}
E_{\Omega,k} = k\gamma +\sum_p \left( k\log \left(1 -\frac1{p}\right)+ \frac{k}{p-1} - \sum_{j=2}^k (-1)^j \binom{k}{j} \frac1{p^j-1} \right).
\end{equation*}
\end{cor}

Finally, we remark that it is also possible to apply our results to functions $f(n)= c \log g(n)$, where $g(n)$ are certain multiplicative functions and $c$ is 
a constant. For example, let $g(n)=\tau(n)$. Then the function $f(n)=\frac{\log \tau(n)}{\log 2}$ is in the class ${\cal F}_0$, and our results can 
be applied. We obtain from Theorem \ref{Th_f_additive_lcm} the following asymptotic formula.

\begin{cor} \label{Cor_tau} If $k\in \N$, then
\begin{equation*}
\sum_{n_1,\ldots, n_k\le x} \log \tau([n_1,\ldots,n_k]) = k (\log 2) x^k \log \log x + A_k x^k
+ k(\log 2)  \sum_{j=1}^N \frac{a_j x^k}{(\log x)^j} +O\left(\frac{x^k}{(\log x)^{N+1}}\right),
\end{equation*}
for every fixed $N\ge 1$, where the constant $A_k$ is given by
\begin{equation*}
A_k = k C - \sum_{j=2}^k (-1)^j \binom{k}{j} D_j, 
\end{equation*}
the sum being $0$ if $k=1$, with
\begin{equation*}
C = (\log 2) \gamma +\sum_p \left( (\log 2) \log \left(1 -\frac1{p}\right) + \left(1-\frac1{p}\right) \sum_{\nu=1}^{\infty} \frac{\log (\nu+1)}{p^{\nu}} \right)
\end{equation*}
and 
\begin{equation*}
D_j =\sum_p  \left(1-\frac1{p^j}\right) \sum_{\nu=1}^{\infty} \frac{\log (\nu+1)}{p^{\nu j}} \quad (2\le j\le k). 
\end{equation*}
\end{cor}

In the special case $k=1$, the result of Corollary \ref{Cor_tau} has been obtained by Hassani \cite[Th.\ 1.1]{Has2016}, and a weaker asymptotic formula is given in 
\cite[Problem 6.12]{DeKLuc2012}.

\subsection{The class ${\cal F}_1$ of Alladi-Erd\H{o}s-type functions} \label{Sect_F_1}

Let ${\cal F}_1$ denote the class of additive functions $f:\N \to \C$ such that $f(p)=p$ for every prime $p$ and $f(p^{\nu}) \ll \nu^{\ell}p^\nu$ holds uniformly 
for the primes $p$ and $\nu \ge 2$, where $\ell \in \N_0$ is some integer. For example, the functions $A_{\ell}(n)= \sum_{p^\nu \mid \mid n} \nu^{\ell} p$ and 
$\widetilde{A}_{\ell}(n)= \sum_{p^\nu \mid \mid n} \nu^{\ell} p^\nu$ with  
$\ell \in \N_0$ are in ${\cal F}_1$. In particular, $A_1(n)=A(n):= \sum_{p^\nu \mid \mid n} \nu p$  and $A_0(n)=A^*(n): = \sum_{p \mid n} p$ are 
the Alladi-Erd\H{o}s functions.  

It is known that   
\begin{equation} \label{est_Alladi_Erdos}
\sum_{n\le x} A(n)= \sum_{\substack{p^\nu \le x\\ \nu \ge 1}} p \left\lfloor \frac{x}{p^\nu}\right\rfloor = \frac{\pi^2 x^2}{12 \log x} + O\left(\frac{x^2}{(\log x)^2}\right),    
\end{equation}
which can be proved by using a strong form of the prime number theorem. See \cite{AllErd1977}, \cite[p.\ 62, 467]{Ten2015}. Also see \cite[Th.\ 2]{Woo2007} for 
a simple approach leading to a slightly weaker error term. The same formula \eqref{est_Alladi_Erdos} holds for $\sum_{n\le x} A^*(n)= \sum_{p\le x} p \left\lfloor \frac{x}{p} \right\rfloor$. 

We point out the following result.

\begin{theorem} \label{Th_f_additive_1} Let $f$ be a function in class ${\cal F}_1$. Then
\begin{equation*}
\sum_{n\le x} f(n) =  \frac{\pi^2 x^2}{12 \log x} + O\left(\frac{x^2}{(\log x)^2}\right).
\end{equation*}
\end{theorem}

Next we deduce the corresponding estimates for the sums $G_{f,k}(x)$ with $k\ge 2$.

\begin{theorem} \label{Th_f_additive_gcd_1} Let $f$ be a function in class ${\cal F}_1$. Then
\begin{equation*}
\sum_{n_1,n_2\le x} f((n_1,n_2)) = x^2 \log \log x + \overline{D}_{f,2} x^2 +  O\left(\frac{x^2}{\log x}\right),
\end{equation*}
where 
\begin{equation*}
\overline{D}_{f,2} = \gamma + \sum_p   \left(  \log \left(1-\frac1{p}\right) + \left(1-\frac{1}{p^2}\right)  \sum_{\nu=1}^{\infty} \frac{f(p^{\nu})}{p^{2\nu}}\right),
\end{equation*}
and if $k\ge 3$, then
\begin{equation*}
\sum_{n_1,\ldots,n_k\le x} f((n_1,\ldots,n_k)) = \overline{D}_{f,k} x^k + \begin{cases} O(x^{k-1}), & \text{ if $k\ge 4$}, \\ O(x^2\log \log x), & \text{ if $k=3$},
\end{cases}
\end{equation*}
where $\overline{D}_{f,k} =D_{f,k}$ defined by \eqref{D_f_k}.
\end{theorem}

\begin{cor} The estimate of Theorem \ref{Th_f_additive_gcd_1} holds for the function $f(n)=A_{\ell}(n)$ \textup{($\ell \in \N$)} with the constants
\begin{equation*}
\overline{D}_{A_{\ell},2} = \gamma + \sum_p  \left( \log \left(1 -\frac1{p}\right)  +  \frac1{p} \left(1-\frac1{p^2}\right)^{-\ell} \ \sum_{\nu=0}^{\ell -1 } \frac{\langle {\ell \atop \nu} \rangle }{p^{2\nu}} \right),
\end{equation*}
\begin{equation*}
\overline{D}_{A_{\ell},k} = \sum_p \frac1{p^{k-1}} \left(1-\frac1{p^k}\right)^{-\ell} \ \sum_{\nu=0}^{\ell -1} 
\frac{\langle {\ell \atop \nu} \rangle }{p^{\nu k}}  \quad (k\ge 3),
\end{equation*}
the inner sums being considered $1$ if $\ell=0$.

In particular, it holds for the function $A(n)=A_1(n)$ with
\begin{equation*}
\overline{D}_{A,2} = \gamma + \sum_p  \left( \log \left(1 -\frac1{p}\right)  +  \frac{p}{p^2-1} \right) \approx 0.4971,
\end{equation*}
\begin{equation*}
\overline{D}_{A,k} = \sum_p \frac{p}{p^k-1} \quad (k\ge 3),
\end{equation*}
and it holds for the function $A^*(n)=A_0(n)$ with the constants $\overline{D}_{A^*,2}=M$ the Mertens constant, and 
$\overline{D}_{A^*,k} = \sum_p \frac1{p^{k-1}}$  \textup{($k\ge 3$)}.
\end{cor}

Now we obtain the estimate for the sums $L_{f,k}(x)$ involving the lcm of integers.

\begin{theorem} \label{Th_f_additive_lcm_1} Let $f$ be a function in class ${\cal F}_1$. Then for every $k\in \N$,
\begin{equation*}
\sum_{n_1,\ldots,n_k\le x} f([n_1,\ldots,n_k]) =  \frac{\pi^2 k x^{k+1}}{12 \log x} + O\left(\frac{x^{k+1}}{(\log x)^2}\right).
\end{equation*}
\end{theorem}

\subsection{The function $B(n)$} \label{Sect_B}

In this section we consider the function $B(n)=A(n)-A^*(n)$, where $B(p)=0$ for every prime $p$. It would be possible to define and study here another class of additive functions $f$  with $f(p)=0$ for the primes $p$ and with adequate order conditions on $f(p^{\nu})$ with $\nu \ge 2$, but we confine ourselves to the function $B(n)$.  

Alladi and Erd\H{o}s \cite[Th.\ 1.5]{AllErd1977} proved that
\begin{equation*}
\sum_{n \le x} B(n) = x \log \log x + O(x).
\end{equation*}

We improve this estimate in the following way.

\begin{theorem} \label{th:B_1} We have
\begin{equation*}
\sum_{n \le x} B(n) = x \log \log \sqrt{x} + F x  + O \left( \frac{x}{\log x} \right),
\end{equation*}
where 
\begin{equation*}
F:= \gamma + \sum_p \left( \log \left( 1 - \frac{1}{p} \right) + \frac{1}{p-1} \right) \approx 1.0346.
\end{equation*}
\end{theorem}

For the sums involving the gcd and lcm we have the following results.

\begin{theorem} \label{Th_B_gcd}
Let $k \in \N$, $k\ge 2$ be fixed. Then
\begin{equation*}
\sum_{n_1,\ldots, n_k \le x} B \left( (n_1,\dotsc,n_k) \right)  = x^k \sum_p \frac{1}{p^{k-1}(p^k-1)}  +   
\begin{cases} O(x^{k-1}), & \text{ if $k\ge 3$}, \\ O(x\log \log x), & \text{ if $k=2$}.
\end{cases}         
\end{equation*}
\end{theorem}

\begin{theorem} \label{Th_B_lcm}
Let $k \in \N$, $k\ge 2$ be fixed. Then
\begin{equation*}
\sum_{n_1,\dots,n_k \le  x} B\left( \left[ n_1,\dotsc,n_k \right]  \right)  = k x^k \log \log \sqrt{x} + H_k x^k  + O \left( \frac{x^k}{\log x} \right)
\end{equation*}
where 
\begin{equation*}
H_k := k F + \sum_{j=2}^k (-1)^{j-1} {k \choose j} \sum_p \frac{1}{p^{j-1}(p^j-1)}
\end{equation*}
and $F$ is the constant defined in Theorem~{\rm~\ref{th:B_1}}. For instance
\begin{small}
\begin{center}
\begin{tabular}{cccccccccc}
$k$ & $2$ & $3$  & $4$ & $5$ & $6$ & $7$ & $8$ & $9$ & $10$ \\
$H_k$ & $1.816$ & $2.435$ & $2.907$ & $3.255$ & $3.5004$ & $3.658$ & $3.74$ & $3.758$ & $3.719$
\end{tabular}
\end{center}
\end{small}
\end{theorem}

\section{Preliminaries to the proofs} \label{Sect_Prelim}

First we give the proof of the next identity, already mentioned in Section \ref{Sect_Motivation}.

\begin{lemma} \label{Lemma_gcd} Let $f$ be an arithmetic function and let $k\in \N$. Then  
\begin{equation*}
G_{f,k}(x):= \sum_{n_1,\ldots,n_k\le x} f((n_1,\ldots,n_k))  =  \sum_{d\le x} (\mu*f)(d) \lfloor x/d \rfloor^k.
\end{equation*}
\end{lemma}

\begin{proof} Using that $f(n)= \sum_{d\mid n} (\mu*f)(d)$, one has 
\begin{equation*}
G_{f,k}(x) = \sum_{n_1,\ldots,n_k\le x} \sum_{d\mid (n_1,\ldots,n_k)} (\mu*f)(d) = \sum_{n_1=dj_1,\ldots, n_k=dj_k\le x}  (\mu*f)(d) 
\end{equation*}
\begin{equation*}
= \sum_{d\le x} (\mu*f)(d)  \sum_{j_1,\ldots,j_k\le x/d} 1 = \sum_{d\le x} (\mu*f)(d) \lfloor x/d \rfloor^k.
\end{equation*}
\end{proof}

\subsection{Properties of additive functions}

The next result is well-known.

\begin{lemma} \label{Lemma_Mobius}  If $f$ is an additive function, then 
\begin{equation*}
(\mu*f)(n) = \begin{cases} f(p^{\nu}) - f(p^{\nu-1}), & \text{ if $n=p^{\nu}$ ($\nu \ge 1$)}, \\ 0, & \text{ otherwise.}
\end{cases}
\end{equation*}
\end{lemma}

\begin{proof} We have
\begin{equation*}
f(n) = \sum_{p^\nu \mid \mid n} f(p^\nu) = \sum_{\substack{p^\nu \mid n\\ \nu \ge 1}} (f(p^{\nu}) - f(p^{\nu-1})) =\sum_{d\mid n} g(d),
\end{equation*}
where
\begin{equation*}
g(d) = \begin{cases} f(p^{\nu}) - f(p^{\nu-1}), & \text{ if $d=p^{\nu}$ ($\nu \ge 1$)}, \\ 0, & \text{ otherwise.}
\end{cases}
\end{equation*}

Hence $f=g*\1$ and by M\"obius inversion we have $\mu*f=g$.
\end{proof}

We will use the following identity.

\begin{lemma} \label{Lemma_additiv_gcd} Let $f$ be an additive function and let $k\in \N$. Then 
\begin{equation*}
\sum_{n_1,\ldots,n_k\le x} f((n_1,\ldots,n_k)) = \sum_{\substack{p^{\nu} \le x\\ \nu \ge 1}} \left(f(p^{\nu}) - f(p^{\nu-1})\right) \left\lfloor \frac{x}{p^{\nu}} \right\rfloor^k
\end{equation*}
\end{lemma}

\begin{proof} Follows at once by using Lemmas \ref{Lemma_gcd} and \ref{Lemma_Mobius}.
\end{proof}

\subsection{An identity involving the Eulerian numbers} \label{Sect_Eulerian_numbers}

Let $\langle {n \atop k} \rangle $ denote the (classical) Eulerian numbers, defined as the number of permutations $h\in S_n$ with $k$ descents. 
Here a number $i$ is called a descent of $h$ if $h(i) > h(i + 1)$. In the paper we use the identity
\begin{equation} \label{series_id}
\sum_{k=0}^{\infty} k^n x^k= \frac{x}{(1-x)^{n+1}} \sum_{k=0}^{n-1} \left\langle {n \atop k} \right\rangle x^k \quad (n\in \N_0, |x|<1),
\end{equation}
where in the RHS the sum is considered to be $1$ if $n=0$. 

Note that $\langle {n\atop 0}\rangle = \langle {n\atop n-1}\rangle = 1$, $\langle {n\atop n}\rangle = 0$ ($n\ge 1$), and the Eulerian numbers  
have the symmetry property $ \langle {n \atop k} \rangle = \langle {n \atop n-k-1} \rangle$ ($n\ge 1$, $k\ge 0$) and satisfy the recurrence relation 
\begin{equation*}
\left\langle {n \atop k} \right\rangle = (k+1) \left\langle {n-1 \atop k} \right\rangle + (n-k) \left\langle {n-1
\atop k-1} \right\rangle \quad (n\ge 1, k\ge 0),
\end{equation*}
where by convention, $\langle {0\atop 0} \rangle =1$ and $\langle {n\atop k} \rangle=0$ for $k<0$ and $n\ge 1$. See, e.g., \cite[Ch.\ 6]{GKP1994} and 
\cite[Ch.\ 1]{Pet2015}.

We deduce the following estimates.

\begin{lemma} \label{Lemma_est_Euler_numbers} For fixed $k,\ell\in \N$,
\begin{equation} \label{est_1}
\sum_{n=1}^{\infty} \frac{n^{\ell}}{p^{n k}} \ll \frac1{p^k} \quad  \text{ as $p\to \infty$},
\end{equation}
and
\begin{equation} \label{est_2}
\sum_{n=2}^{\infty} \frac{n^{\ell}}{p^{nk}} \ll \frac1{p^{2k}} \quad  \text{ as $p\to \infty$}.
\end{equation}
\end{lemma}

\begin{proof} According to identity \eqref{series_id},
\begin{equation*}
\sum_{n=1}^{\infty} \frac{n^{\ell}}{p^{nk}} = \frac{1/p^{k}}{(1-1/p^{k})^{\ell+1}} \sum_{n=0}^{\ell-1} \left\langle {\ell \atop n} \right\rangle \frac1{p^{nk}}
\end{equation*}
\begin{equation*}
\ll  \frac1{p^{k}} \sum_{n=0}^{\ell-1} \left\langle {\ell \atop n} \right\rangle \frac1{p^{nk}} \ll \frac1{p^{k}},
\end{equation*}
where $\ell$ is fixed, we have finitely many terms, and the largest term -- with respect to $p$ -- of the sum is that for $n=0$. This gives \eqref{est_1}. Similarly,
\begin{equation*}
\sum_{n=2}^{\infty} \frac{n^{\ell}}{p^{nk}} = - \frac1{p^{k}} +\frac{1/p^{k}}{(1-1/p^{k})^{\ell+1}} \sum_{n=0}^{\ell-1} \left\langle {\ell \atop n} \right\rangle \frac1{p^{nk}}
\end{equation*}
\begin{equation*}
\ll  - \frac1{p^{k}}\left(1-\frac1{p^{k}} \right)^{\ell+1} +\frac1{p^{k}} \sum_{n=0}^{\ell-1} \left\langle {\ell \atop n} \right\rangle \frac1{p^{nk}}
\end{equation*}
\begin{equation*}
=  - \frac1{p^{k}}+ \frac{\ell+1}{p^{2k}} + \cdots + \frac{(-1)^{\ell}}{p^{(\ell+2)k}} + \frac1{p^{k}} \sum_{n=0}^{\ell-1} \left\langle {\ell \atop n} \right\rangle \frac1{p^{nk}}
\ll \frac1{p^{2k}},
\end{equation*}
where, since $\langle {\ell \atop 0} \rangle=1$, the term $1/p^{k}$ cancels out, giving \eqref{est_2}.
\end{proof}

\subsection{Estimates of certain sums}

The estimates of Lemma \ref{Lemma_est_Euler_numbers} are not sufficient for our proofs. We need good estimates on the sums $\sum_{n>z} \frac{n^{\ell}}{p^{nk}}$, where
$z\ge 1$ is a real number.

\begin{lemma}  \label{le:toth}
i) Let $k,\ell \in \N$, $p$ be a prime and $z \geqslant 1$ be a real number satisfying $z > \frac{\ell \max(1,\log z)}{k \log p}$. Then
\begin{equation} \label{est_main}
\sum_{n > z} \frac{n^\ell}{p^{n k}} \le  \frac{z^\ell}{p^{kz}} \left( \frac{1}{k \log p - \ell/z} + 1 \right).
\end{equation}

ii) In particular, if $k \log p > \max \left( \frac{\ell \max(1,\log z)}{z} \, , \, \frac{\ell}{z} + \frac{1}{2} \right)$, then
\begin{equation} \label{estimate_ell}
\sum_{n > z} \frac{n^\ell}{p^{n k}} \le  \frac{3 z^\ell}{p^{kz}}.
\end{equation}
\end{lemma}

Note that \eqref{est_1} and \eqref{est_2} are special cases of \eqref{estimate_ell}. The proof of Lemma \ref{le:toth} uses the following 
van der Corput type bound.

\begin{lemma} \label{le:vdc_type}
Let $a<b$ be real numbers and $g \in C^1 [a,b]$ such that $g \ge 0$, $g^{\, \prime}$ is non-decreasing and there exists 
$\lambda_1 >0$ such that, for all $x \in [a,b]$, we have $g^{\, \prime}(x) \ge  \lambda_1$. Then
\begin{equation*}
\int_a^b e^{-g(t)} \, \mathrm{d}t \le \frac{e^{-g(a)}}{\lambda_1}.
\end{equation*}
\end{lemma}

\begin{proof}[Proof of Lemma~\ref{le:vdc_type}]
Use the following inequality due to Ostrowski. See, e.g., \cite[(3.7.35)]{Mit1970}. Let $a<b$ be real numbers and $F$ be a real-valued function, $G$ be a complex-valued 
function, both integrable on $[a,b]$ such that $F$ is monotone, $F(a) F(b) \ge  0$ and $|F(a)| \ge |F(b)|$. Then
\begin{equation*}
\left | \int_a^b F(t) \, G(t) \, \mathrm{d}t \right | \le |F(a)| \, \max_{a\le x\le b} \left | 
\int_a^x G(t) \, \mathrm{d}t \right |.
\end{equation*}

Select $F(t) = 1/g^{\, \prime}(t)$ and $G(t) = g^{\, \prime}(t) e^{-g(t)}$  and proceeding as in van der Corput's first derivative test, we have
\begin{equation*}
\int_a^b e^{-g(t)} \, \textrm{d}t \le  \frac{1}{g^{\, \prime}(a)} \max_{a\le x\le b} \ \int_a^x g'(t) e^{-g(t)} \, \textrm{d}t \le  \frac{1}{\lambda_1} \max_{a\le x\le b}\ \left( e^{-g(a)} - e^{-g(x)} \right) \le  \frac{e^{-g(a)}}{\lambda_1}
\end{equation*}
as asserted.
\end{proof}

\begin{proof}[Proof of Lemma~\ref{le:toth}]
Apply Lemma~\ref{le:vdc_type} with $g(t):= tk \log p - \ell \log t$, for which $g^{\, \prime} (t) = k \log p - \frac{\ell}{t} \ge k \log p - \frac{\ell}{z}$ 
for all $t \ge z$. Note that the condition $z > \frac{\ell \max(1,\log z)}{k \log p}$ ensures that both $g(z) > 0$ and $g^{\, \prime} (z) > 0$. Lemma~\ref{le:vdc_type} 
then yields
\begin{equation*}
   \sum_{n > z} \frac{n^\ell}{p^{n k}} = \sum_{n > z} e^{- g(n)} \le \int_z^\infty e^{-g(t)} \, \textrm{d}t \ + \; e^{-g(z)}  
 \end{equation*}
 \begin{equation*}
    \le \frac{e^{-g(z)}}{k \log p - \ell/z} \ + \; e^{-g(z)}
   = \frac{z^\ell}{p^{kz}} \left( \frac{1}{k \log p - \ell/z} \ + \; 1 \right).
\end{equation*}
The second part of the lemma follows by noticing that, if $k \log p > \frac{\ell}{z} + \frac{1}{2}$, then $\frac{1}{k \log p - \ell/z} < 2$.
\end{proof}

\begin{lemma} Let $k,\ell \in \N$ be fixed. We have, as $x\to \infty$,
\begin{equation} \label{S_k}
\sum_{p \le x}  \ \sum_{\nu > \log x/ \log p} \frac{\nu^{\ell}}{p^{\nu k}} \ll \frac1{x^{k-1}(\log x)^2},
\end{equation}
\begin{equation} \label{S_1}
\sum_{p \le x^{\1/2}}  \ \sum_{\nu > \log x/ \log p} \frac{\nu^{\ell}}{p^{\nu}} \ll \frac1{x^{1/2} (\log x)^2}.
\end{equation}
\end{lemma}

\begin{proof} Assume $x \ge e^{2 \ell/k}$ and use Lemma~\ref{le:toth} with $z = \frac{\log x}{\log p}$. Note that 
the condition $z > \frac{\ell}{k\log p}$ is fulfilled as soon as $x > e^{\ell/k}$. Now Lemma~\ref{le:toth} yields
\begin{equation*}
   \sum_{\nu > \log x/ \log p} \frac{\nu^\ell}{p^{\nu k}} 
   \le \frac{1}{x^k} \left( \frac{\log x}{\log p}\right)^{\ell} 
   \left(k \log p - \frac{\ell \log p}{\log x}\right)^{-1} 
 \end{equation*}
 \begin{equation*}
 = \frac{1}{kx^k} \frac{(\log x)^{\ell}}{(\log p)^{\ell+1}} \left( 1 - \frac{\ell}{k\log x}\right)^{-1} 
    \le \frac{2(\log x)^{\ell}}{kx^k(\log p)^{\ell+1}},
\end{equation*}
where we used the inequality $(1-a)^{-1} \le 2$ provided that $0 \le a \le \frac{1}{2}$. Hence
\begin{equation} \label{S}
S:= \sum_{p \le x}  \ \sum_{\nu > \log x/ \log p} \frac{\nu^{\ell}}{p^{\nu k}} \ll 
\frac{(\log x)^\ell}{x^k} \sum_{p \le x} \frac{1}{(\log p)^{\ell+1}}.
\end{equation}

We also have
\begin{equation} \label{est_split}
\sum_{p \le x} \frac{1}{(\log p)^{\ell+1}} = \left( \sum_{p \le  \sqrt{x}} + \sum_{\sqrt{x} < p \le x}\right) \frac{1}{(\log p)^{\ell+1}}
\le \frac{\pi(\sqrt x)}{(\log 2)^{\ell+1}} + \frac{2^{\ell + 1} \pi(x)}{(\log x)^{\ell+1}} \ll \frac{x}{(\log x)^{\ell+2}},
\end{equation}
using the Chebysev estimate $\pi(x)\ll x/\log x$.

From \eqref{S} and \eqref{est_split} we obtain that
\begin{equation*} 
S\ll  \frac{1}{x^{k-1}(\log x)^2}.
\end{equation*}

This proves estimate \eqref{S_k}. The proof of \eqref{S_1} is similar.
\end{proof}

\section{Proofs of the asymptotic formulas} \label{Sect_Proofs}

\subsection{Proofs of the results in Section \ref{Sect_F_0}}

\begin{proof}[Proof of Theorem {\rm \ref{Th_f_additive}}]
Using Lemma \ref{Lemma_additiv_gcd} for $k=1$, and that $f(p)=1$,
\begin{equation*}
\sum_{n\le x } f(n)  = \sum_{\substack{p^{\nu} \le x\\ \nu \ge 1}} \left(f(p^{\nu}) - f(p^{\nu-1})\right) \left\lfloor \frac{x}{p^{\nu}} \right\rfloor
\end{equation*}
\begin{equation} \label{S_S}
 = \sum_{p\le x} \left\lfloor \frac{x}{p}\right\rfloor + \sum_{\substack{p^{\nu} \le x\\ \nu \ge 2}} \left(f(p^{\nu}) - f(p^{\nu-1})\right) \left\lfloor \frac{x}{p^{\nu}} \right\rfloor =: S_1 + S_2,
\end{equation}
where $S_1:= \sum_{p\le x} \left\lfloor \frac{x}{p}\right\rfloor = \sum_{n\le x} \omega(n)$ and Saffari's estimate \eqref{est_Saffari} can be applied.

We estimate the sum
\begin{equation*}
S_2:= \sum_{\substack{p^{\nu} \le x\\ \nu \ge 2}} \left(f(p^{\nu}) - f(p^{\nu-1})\right) \left\lfloor \frac{x}{p^{\nu}} \right\rfloor
\end{equation*}
\begin{equation} \label{terms}
=  x \sum_{\substack{p^{\nu} \le x\\ \nu \ge 2}} \frac{f(p^{\nu}) - f(p^{\nu-1})}{p^{\nu}} - 
\sum_{\substack{p^{\nu} \le x\\ \nu \ge 2}} \left(f(p^{\nu}) - f(p^{\nu-1})\right) (x/p^{\nu} - \lfloor x/ p^{\nu} \rfloor).
\end{equation}

By the definition of the class ${\cal F}_0$ we have $f(p^{\nu}) - f(p^{\nu-1})\ll \nu^{\ell}$, uniformly for the primes $p$ and $\nu \ge 2$, for some $\ell \in \N_0$, and the second sum in 
\eqref{terms} is 
\begin{equation*}
\ll \sum_{\substack{p^{\nu} \le x\\ \nu \ge 2}} \nu^{\ell} = \sum_{p\le x^{1/2}} \sum_{2\le \nu \le \log x/\log p } \nu^{\ell}
\ll  (\log x)^{\ell +1} \sum_{p\le x^{1/2}}  \frac1{(\log p)^{\ell +1}} \ll \frac{x^{1/2}}{\log x},
\end{equation*}
by using \eqref{est_split}.

We also have
\begin{equation*}
\sum_p \sum_{\nu =2}^{\infty} \frac{|f(p^{\nu}) - f(p^{\nu-1})|}{p^{\nu}} \ll \sum_p \sum_{\nu =2}^{\infty} \frac{\nu^{\ell}}{p^{\nu}}
\ll \sum_p  \frac1{p^2} <\infty,
\end{equation*}
by estimate \eqref{est_2} with $k=1$ or by \eqref{estimate_ell} applied for $z=2$, $k=1$. Hence the series 
\begin{equation*}
\sum_p \sum_{\nu =2}^{\infty} \frac{f(p^{\nu}) - f(p^{\nu-1})}{p^{\nu}}  = \sum_p \left( \left(1-\frac1{p}\right)\sum_{\nu=1}^{\infty} \frac{f(p^{\nu})}{p^{\nu}} -\frac1{p} \right)=: B_f
\end{equation*}
is absolutely convergent. Also
\begin{equation*}
\sum_{\substack{p^{\nu} \le x\\ \nu \ge 2}} \frac{f(p^{\nu}) - f(p^{\nu-1})}{p^{\nu}} = 
\sum_{p\le x^{1/2}} \sum_{\substack{\nu \ge 2\\ p^{\nu} \le x}} \frac{f(p^{\nu}) - f(p^{\nu-1})}{p^{\nu}} 
\end{equation*}
\begin{equation*}
= \sum_{p\le x^{1/2}} \left( \sum_{\nu =2}^{\infty} \frac{f(p^{\nu}) - f(p^{\nu-1})}{p^{\nu}} 
- \sum_{\nu > \log x/ \log p}  \frac{f(p^{\nu}) - f(p^{\nu-1})}{p^{\nu}} \right)
\end{equation*}
\begin{equation*}
= \sum_p \sum_{\nu =2}^{\infty} \frac{f(p^{\nu}) - f(p^{\nu-1})}{p^{\nu}} 
-\sum_{p>x^{1/2}} \sum_{\nu =2}^{\infty} \frac{f(p^{\nu}) - f(p^{\nu-1})}{p^{\nu}}
\end{equation*}
\begin{equation} \label{last_sum}
- \sum_{p\le x^{1/2}} \sum_{\nu > \log x/ \log p}  \frac{f(p^{\nu}) - f(p^{\nu-1})}{p^{\nu}}, 
\end{equation}
where
\begin{equation*}
\sum_{p>x^{1/2}} \sum_{\nu =2}^{\infty} \frac{f(p^{\nu}) - f(p^{\nu-1})}{p^{\nu}} \ll \sum_{p>x^{1/2}} \sum_{\nu=2}^{\infty} 
\frac{\nu^{\ell}}{p^{\nu}}
\ll \sum_{p>x^{1/2}} \frac1{p^2} \ll \sum_{n>x^{1/2}} \frac1{n^2} \ll \frac1{x^{1/2}},
\end{equation*}
again by \eqref{est_2} or by \eqref{estimate_ell}.

Furthermore, the double sum in \eqref{last_sum} is $\ll \frac1{x^{1/2}(\log x)^2}$ by estimate \eqref{S_1}.

Putting these altogether gives
\begin{equation} \label{est_final_S_2}
S_2= B_f + O(x^{1/2}).
\end{equation}

Now the proof is complete by \eqref{S_S}, \eqref{est_Saffari} and  \eqref{est_final_S_2}.
\end{proof}

\begin{proof}[Proof of Corollary {\rm \ref{Cor_Omega_ell}}]
Follows from Theorem \ref{Th_f_additive} and identity \eqref{series_id}.
\end{proof}

\begin{proof}[Proof of Corollary {\rm \ref{Cor_T_ell}}]
Follows from Theorem \ref{Th_f_additive} and the known identity 
\begin{equation}  \label{id_binom}
\sum_{k=0}^{\infty} \binom{n+k-1}{k} x^k= \frac1{(1-x)^n} \quad (n\in \N_0, |x|<1).
\end{equation}
\end{proof}

\begin{proof}[Proof of Theorem {\rm \ref{Th_f_additive_gcd}}]
Using Lemma \ref{Lemma_additiv_gcd},
\begin{equation*}
\sum_{n_1,\ldots,n_k\le x} f((n_1,\ldots,n_k)) = \sum_{\substack{p^{\nu} \le x\\ \nu \ge 1}} \left(f(p^{\nu}) - f(p^{\nu-1})\right) \left\lfloor \frac{x}{p^{\nu}} \right\rfloor^k,
\end{equation*} 
where
\begin{equation*}
\left\lfloor \frac{x}{p^{\nu}} \right\rfloor^k = \left( \frac{x}{p^{\nu}} + O(1) \right)^k = \left(\frac{x}{p^{\nu}}\right)^k + O\left(\left(\frac{x}{p^{\nu}}\right)^{k-1}\right).
\end{equation*}

We deduce
\begin{equation*}
\sum_{n_1,\ldots,n_k\le x} f((n_1,\ldots,n_k)) = x^k \sum_{\substack{p^{\nu} \le x\\ \nu \ge 1}} \frac{f(p^{\nu}) - f(p^{\nu-1})}{p^{\nu k}} + x^{k-1} R_{f,k}(x),
\end{equation*}
where
\begin{equation*}
R_{f,k}(x) \ll \sum_{\substack{p^{\nu} \le x\\ \nu \ge 1}} \frac{f(p^{\nu}) - f(p^{\nu-1})}{p^{\nu (k-1)}}
\ll \sum_{p\le x}  \sum_{\nu=1}^{\infty} \frac{\nu^{\ell}}{p^{\nu (k-1)}} \ll \sum_{p\le x} \frac1{p^{k-1}}, 
\end{equation*}
using \eqref{estimate_ell} or \eqref{est_1}, which is $\ll \log \log x$ for $k=2$ and is $\ll 1$ for $k\ge 3$.

Here 
\begin{equation*}
\sum_p \sum_{\nu = 1}^{\infty} \frac{|f(p^{\nu}) - f(p^{\nu-1})|}{p^{\nu k}}  \ll  
\sum_p \sum_{\nu =1}^{\infty} \frac{\nu^{\ell}}{p^{\nu k}} \ll \sum_p \frac1{p^k} < \infty,
\end{equation*}
again by \eqref{est_1} or \eqref{estimate_ell}, where $k\ge 2$, hence the series 
\begin{equation*}
\sum_p \sum_{\nu = 1}^{\infty} \frac{f(p^{\nu}) - f(p^{\nu-1})}{p^{\nu k}} = 
\sum_p \left(1-\frac1{p^k}\right) \sum_{\nu = 1}^{\infty} \frac{f(p^{\nu})}{p^{\nu k}} =: D_{f,k}
\end{equation*}
is absolutely convergent.

Furthermore,
\begin{equation*}
\sum_{\substack{p^{\nu} \le x\\ \nu \ge 1}} \frac{f(p^{\nu}) - f(p^{\nu-1})}{p^{\nu k}} = \sum_{p\le x} \sum_{1\le \nu \le \log x/\log p} 
\frac{f(p^{\nu}) - f(p^{\nu-1})}{p^{\nu k}}
\end{equation*}
\begin{equation*}
= \sum_{p\le x} \left(\sum_{\nu =1}^{\infty} \frac{f(p^{\nu}) - f(p^{\nu-1})}{p^{\nu k}} -  \sum_{\nu  > \log x/\log p} 
\frac{f(p^{\nu}) - f(p^{\nu-1})}{p^{\nu k}} \right)
\end{equation*}
\begin{equation*}
= \sum_p \sum_{\nu =1}^{\infty} \frac{f(p^{\nu}) - f(p^{\nu-1})}{p^{\nu k}} -  \sum_{p>x} \sum_{\nu=1}^{\infty} 
\frac{f(p^{\nu}) - f(p^{\nu-1})}{p^{\nu k}}  - \sum_{p\le x} \sum_{\nu  > \log x/\log p} 
\frac{f(p^{\nu}) - f(p^{\nu-1})}{p^{\nu k}}.
\end{equation*}

Here
\begin{equation*}
\sum_{p>x} \sum_{\nu =1}^{\infty} \frac{f(p^{\nu}) - f(p^{\nu-1})}{p^{\nu k}} \ll
\sum_{p>x} \sum_{\nu =1}^{\infty} \frac{\nu^{\ell}}{p^{\nu k}} \ll \sum_{p>x} \frac1{p^k} \ll \sum_{n>x} \frac1{n^k} \ll \frac1{x^{k-1}}, 
\end{equation*}
again by \eqref{est_1} or \eqref{estimate_ell}. Also,
\begin{equation*}
\sum_{p\le x} \sum_{\nu  > \log x/\log p} \frac{f(p^{\nu}) - f(p^{\nu-1})}{p^{\nu k}} 
\ll \sum_{p\le x} \sum_{\nu > \log x/\log p} \frac{\nu^{\ell}}{p^{\nu k}} \ll \frac1{x^{k-1}(\log x)^2}, 
\end{equation*}
by \eqref{S_k}. This completes the proof.
\end{proof}

\begin{proof}[Proof of Corollary {\rm \ref{Cor_Omega_ell_gcd}}] 
Follows from Theorem \ref{Th_f_additive_gcd} and identities \eqref{series_id}, \eqref{id_binom}.
\end{proof}

\begin{proof}[Proof of Theorem {\rm \ref{Th_f_additive_lcm}}] 

By Proposition \ref{Prop_key} and by symmetry we have
\begin{equation*}
L_{f,k}(x) = \sum_{n_1,\ldots,n_k\le x} f([n_1,\ldots,n_k]) = \sum_{n_1,\ldots,n_k\le x} \sum_{1\le j\le k} (-1)^{j-1} 
\sum_{1\le i_1< \ldots < i_j \le  k} f((n_{i_1},\ldots,n_{i_j}))
\end{equation*}
\begin{equation*}
= \sum_{1\le j\le k} (-1)^{j-1} \binom{k}{j} \sum_{n_1,\ldots,n_k\le x}  f((n_1,\ldots,n_j))
\end{equation*}
\begin{equation*}
= \sum_{1\le j\le k} (-1)^{j-1} \binom{k}{j} \sum_{n_1,\ldots,n_j\le x} f((n_1,\ldots,n_j)) \sum_{n_{j+1},\ldots,n_k\le x} 1,
\end{equation*}
where the last sum is $\lfloor x \rfloor^{k-j}= x^{k-j}+ O(x^{k-j-1})$. Therefore, by Theorems \ref{Th_f_additive} and \ref{Th_f_additive_gcd},
\begin{equation*}
L_{f,k}(x)=k \left(\sum_{n_1\le x} f(n_1)\right) \left(x^{k-1}+O(x^{k-2})\right) 
\end{equation*}
\begin{equation*}
+ \sum_{2\le j\le k} (-1)^{j-1}  \binom{k}{j} \left( \sum_{n_1,\ldots,n_j\le x} f((n_1,\ldots,n_j))\right) \left( x^{k-j}+ O(x^{k-j-1})\right) 
\end{equation*}
\begin{equation*}
=k \left(x\log \log x + C_f x + x \sum_{j=1}^N \frac{a_j}{(\log x)^j} + O\left(\frac{x}{(\log x)^{N+1}}\right) \right) \left(x^{k-1}+O(x^{k-2})\right) 
\end{equation*}
\begin{equation*}
+ \sum_{2\le j\le k} (-1)^{j-1}  \binom{k}{j} \left(D_{f,j} x^j + O\left(R_j(x) \right)\right) \left( x^{k-j}+ O(x^{k-j-1})\right), 
\end{equation*}
where $R_j(x)=x^{j-1}$ ($j\ge 3$) and  $R_2(x)= x\log \log x$ ($j=2$). This gives the result.
\end{proof}

\subsection{Proofs of the results in Section \ref{Sect_F_1}}

\begin{proof}[Proof of Theorem {\rm \ref{Th_f_additive_1}}] 
Using Lemma \ref{Lemma_additiv_gcd} for $k=1$, and that $f(p)=p$,
\begin{equation*}
\sum_{n\le x } f(n)  = \sum_{\substack{p^{\nu} \le x\\ \nu \ge 1}} \left(f(p^{\nu}) - f(p^{\nu-1})\right) \left\lfloor \frac{x}{p^{\nu}} \right\rfloor
\end{equation*}
\begin{equation} \label{T_T}
 = \sum_{p\le x} p \left\lfloor \frac{x}{p}\right\rfloor + \sum_{\substack{p^{\nu} \le x\\ \nu \ge 2}} \left(f(p^{\nu}) - f(p^{\nu-1})\right) \left\lfloor \frac{x}{p^{\nu}} \right\rfloor =: T_1 + T_2,
\end{equation}
where $T_1:= \sum_{p\le x} p \left\lfloor \frac{x}{p}\right\rfloor = \sum_{n\le x} A(n)$ and estimate \eqref{est_Alladi_Erdos} can be applied.

If $f\in  {\cal F}_1$, then $f(p^{\nu}) - f(p^{\nu-1})\ll \nu^{\ell}p^{\nu}$ for some $\ell \in \N_0$, and we show that the sum $T_2$ is negligible by comparison:
\begin{equation*}
T_2:= \sum_{\substack{p^{\nu} \le x\\ \nu \ge 2}} \left(f(p^{\nu}) - f(p^{\nu-1})\right) \left\lfloor \frac{x}{p^{\nu}} \right\rfloor
\end{equation*}
\begin{equation*} 
\ll x \sum_{\substack{p^{\nu} \le x\\ \nu \ge 2}} \nu^{\ell} = x \sum_{2\le \nu \le \log x / \log 2} \nu^{\ell} \sum_{p\le x^{1/\nu}} 1
\end{equation*}
\begin{equation*}\
\ll x \sum_{2\le \nu \le \log x / \log 2} \nu^{\ell} \frac{x^{1/\nu}}{\log x^{1/\nu}} \ll \frac{x^{3/2}}{\log x} \sum_{\nu \le \log x / \log 2} \nu^{\ell+1} \ll x^{3/2} (\log x)^{\ell+1}.
\end{equation*}
\end{proof}

\begin{proof}[Proof of Theorem {\rm \ref{Th_f_additive_gcd_1}}] 
Similar to the proof of Theorem \ref{Th_f_additive_gcd}. By using Lemma \ref{Lemma_additiv_gcd},
\begin{equation*}
L_{f,k}(x):= \sum_{n_1,\ldots,n_k\le x} f((n_1,\ldots,n_k)) = \sum_{\substack{p^{\nu} \le x\\ \nu \ge 1}} \left(f(p^{\nu}) - f(p^{\nu-1})\right) \left\lfloor \frac{x}{p^{\nu}} \right\rfloor^k
\end{equation*}
\begin{equation*}
= x^k \sum_{\substack{p^{\nu} \le x\\ \nu \ge 1}} \frac{f(p^{\nu}) - f(p^{\nu-1})}{p^{\nu k}} + x^{k-1} V_{f,k}(x),
\end{equation*}
where
\begin{equation*}
V_{f,k}(x) \ll \sum_{\substack{p^{\nu} \le x\\ \nu \ge 1}} \frac{f(p^{\nu}) - f(p^{\nu-1})}{p^{\nu (k-1)}}
\ll \sum_{p\le x}  \sum_{\nu=1}^{\infty} \frac{\nu^{\ell}}{p^{\nu (k-2)}}.  
\end{equation*}

Case I. If $k\ge 3$, then $V_{f,k}(x) \ll \sum_{p\le x} \frac1{p^{k-2}}$, using \eqref{estimate_ell} or \eqref{est_1}, which is $\ll \log \log x$ for $k=3$ and is $\ll 1$ for $k\ge 4$. 

Also, 
\begin{equation*}
\sum_p \sum_{\nu = 1}^{\infty} \frac{|f(p^{\nu}) - f(p^{\nu-1})|}{p^{\nu k}}  \ll  
\sum_p \sum_{\nu =1}^{\infty} \frac{\nu^{\ell}}{p^{\nu (k-1)}} \ll \sum_p \frac1{p^{k-1}} < \infty,
\end{equation*}
by \eqref{est_1}, where $k\ge 3$, hence the series 
\begin{equation*}
\sum_p \sum_{\nu = 1}^{\infty} \frac{f(p^{\nu}) - f(p^{\nu-1})}{p^{\nu k}} = 
\sum_p \left(1-\frac1{p^k}\right) \sum_{\nu = 1}^{\infty} \frac{f(p^{\nu})}{p^{\nu k}} =: \overline{D}_{f,k}
\end{equation*}
is absolutely convergent.

Furthermore, like in the proof of Theorem \ref{Th_f_additive_gcd}, 
\begin{equation*}
\sum_{\substack{p^{\nu} \le x\\ \nu \ge 1}} \frac{f(p^{\nu}) - f(p^{\nu-1})}{p^{\nu k}} = 
\end{equation*}
\begin{equation*}
= \sum_p \sum_{\nu =1}^{\infty} \frac{f(p^{\nu}) - f(p^{\nu-1})}{p^{\nu k}} -  \sum_{p>x} \sum_{\nu=1}^{\infty} 
\frac{f(p^{\nu}) - f(p^{\nu-1})}{p^{\nu k}}  - \sum_{p\le x} \sum_{\nu  > \log x/\log p} 
\frac{f(p^{\nu}) - f(p^{\nu-1})}{p^{\nu k}},
\end{equation*}
where
\begin{equation*}
\sum_{p>x} \sum_{\nu =1}^{\infty} \frac{f(p^{\nu}) - f(p^{\nu-1})}{p^{\nu k}} \ll
\sum_{p>x} \sum_{\nu =1}^{\infty} \frac{\nu^{\ell}}{p^{\nu (k-1)}} \ll \sum_{p>x} \frac1{p^{k-1}} \ll \sum_{n>x} \frac1{n^{k-1}} \ll \frac1{x^{k-2}}, 
\end{equation*}

Also,
\begin{equation*}
\sum_{p\le x} \sum_{\nu  > \log x/\log p} \frac{f(p^{\nu}) - f(p^{\nu-1})}{p^{\nu k}} 
\ll \sum_{p\le x} \sum_{\nu > \log x/\log p}^{\infty} \frac{\nu^{\ell}}{p^{\nu (k-1)}} \ll \frac1{x^{k-2}(\log x)^2}, 
\end{equation*}
by \eqref{S_k}. 

Case II. If $k=2$, then 
\begin{equation*}
V_{f,2}(x) \ll \sum_{p\le x} \sum_{\nu \le \log x/ \log p} \nu^{\ell} \ll \sum_{p\le x} \left(\frac{\log x}{\log p}\right)^{\ell +1}\ll \frac{x}{\log x} 
\end{equation*}
by \eqref{est_split}. Also,
\begin{equation*}
\sum_{\substack{p^{\nu} \le x\\ \nu \ge 1}} \frac{f(p^{\nu}) - f(p^{\nu-1})}{p^{2\nu}} = \sum_{p \le x}  \frac1{p} + \sum_{\substack{p^{\nu} \le x\\ \nu \ge 2}} \frac{f(p^{\nu}) - f(p^{\nu-1})}{p^{2\nu}},
\end{equation*}
where
\begin{equation*}
\sum_{p\le x} \frac1{p} = \log \log x + M + O\left(\frac1{\log x} \right),    
\end{equation*}
and 
\begin{equation*}
\sum_p \sum_{\nu = 2}^{\infty} \frac{|f(p^{\nu}) - f(p^{\nu-1})|}{p^{2\nu}}  \ll  
\sum_p \sum_{\nu =2}^{\infty} \frac{\nu^{\ell}}{p^{\nu}} \ll \sum_p \frac1{p^2} < \infty,
\end{equation*}
by \eqref{est_1}, hence the series 
\begin{equation*}
\sum_p \sum_{\nu = 2}^{\infty} \frac{f(p^{\nu}) - f(p^{\nu-1})}{p^{2\nu}} = 
\sum_p \left(-\frac1{p}+ \left(1-\frac1{p^2}\right) \sum_{\nu = 2}^{\infty} \frac{f(p^{\nu})}{p^{2\nu }}\right) 
\end{equation*}
is absolutely convergent.

Furthermore, in a similar manner as above, 
\begin{equation*}
\sum_{\substack{p^{\nu} \le x\\ \nu \ge 2}} \frac{f(p^{\nu}) - f(p^{\nu-1})}{p^{2\nu}} 
\end{equation*}
\begin{equation*}
= \sum_p \sum_{\nu =2}^{\infty} \frac{f(p^{\nu}) - f(p^{\nu-1})}{p^{2\nu}} -  \sum_{p>x^{1/2}} \sum_{\nu=2}^{\infty} 
\frac{f(p^{\nu}) - f(p^{\nu-1})}{p^{2\nu}}  - \sum_{p\le x^{1/2}} \sum_{\nu  > \log x/\log p} 
\frac{f(p^{\nu}) - f(p^{\nu-1})}{p^{2\nu}},
\end{equation*}
where
\begin{equation*}
\sum_{p>x^{1/2}} \sum_{\nu =2}^{\infty} \frac{f(p^{\nu}) - f(p^{\nu-1})}{p^{2\nu }} \ll
\sum_{p>x^{1/2}} \sum_{\nu =2}^{\infty} \frac{\nu^{\ell}}{p^{\nu}} \ll \sum_{p>x^{1/2}} \frac1{p^2} \ll \sum_{n>x^{1/2}} \frac1{n^2} \ll \frac1{x^{1/2}}, 
\end{equation*}
and
\begin{equation*}
\sum_{p\le x^{1/2}} \sum_{\nu  > \log x/\log p} \frac{f(p^{\nu}) - f(p^{\nu-1})}{p^{2\nu}} 
\ll \sum_{p\le x^{1/2}} \sum_{\nu > \log x/\log p}^{\infty} \frac{\nu^{\ell}}{p^{\nu}} \ll \frac1{x^{1/2}(\log x)^2}, 
\end{equation*}
by \eqref{S_1}.
\end{proof}

\begin{proof}[Proof of Theorem {\rm \ref{Th_f_additive_lcm_1}}] Similar to the proof of Theorem \ref{Th_f_additive_lcm}. By Proposition \ref{Prop_key} and by symmetry we have
\begin{equation*}
L_{f,k}(x): = \sum_{n_1,\ldots,n_k\le x} f([n_1,\ldots,n_k]) = \sum_{n_1,\ldots,n_k\le x} \sum_{1\le j\le k} (-1)^{j-1} 
\sum_{1\le i_1< \ldots < i_j \le  k} f((n_{i_1},\ldots,n_{i_j}))
\end{equation*}
\begin{equation*}
= \sum_{1\le j\le k} (-1)^{j-1} \binom{k}{j} \sum_{n_1,\ldots,n_k\le x}  f((n_1,\ldots,n_j))
\end{equation*}
\begin{equation*}
= k \left(\sum_{n_1\le x} f(n_1)\right) \left(x^{k-1}+O(x^{k-2})\right) - \binom{k}{2} \left(\sum_{n_1,n_2\le x} f((n_1,n_2)) \right) \left(x^{k-2}+O(x^{k-3})\right) 
\end{equation*}
\begin{equation*}
+ \sum_{3\le j\le k} (-1)^{j-1}  \binom{k}{j} \left( \sum_{n_1,\ldots,n_j\le x} f((n_1,\ldots,n_j))\right) \left( x^{k-j}+ O(x^{k-j-1})\right).
\end{equation*}

Now using Theorems \ref{Th_f_additive_1} and \ref{Th_f_additive_gcd_1} gives the result.
\end{proof}

\subsection{Proofs of the results in Section \ref{Sect_B}}

\begin{proof}[Proof of Theorem {\rm \ref{th:B_1}}]
Since $B(p)=0$ and $B \left( p^\nu \right) - B \left( p^{\nu-1} \right) = p$ as soon as $\nu \ge 2$, Lemma \ref{Lemma_additiv_gcd} with $k=1$ yields
\begin{equation*}
  \sum_{n \le x} B(n) = \sum_{\substack{p^\nu \le x \\ \nu \ge 2}} p \left \lfloor \frac{x}{p^\nu} \right \rfloor 
  = \sum_{p \le \sqrt{x}} p \sum_{2 \le \nu 
  \le \frac{\log x}{\log p}} \left\lfloor  \frac{x}{p^\nu} \right\rfloor 
\end{equation*}
\begin{equation*}
 = x \sum_{p \le \sqrt{x}} p \sum_{2 \le \nu\le \frac{\log x}{\log p}} \frac{1}{p^\nu} - \sum_{p \le \sqrt{x}} p \sum_{2 \le \nu \le \frac{\log x}{\log p}} \left(\frac{x}{p^\nu} - \left\lfloor  
 \frac{x}{p^\nu} \right\rfloor \right) 
 \end{equation*}
\begin{equation*}
 := x \Sigma_1 - \Sigma_2,
\end{equation*}
with
\begin{equation*}
   \Sigma_1 = \sum_{p \le \sqrt{x}} p \sum_{\nu = 2}^\infty \frac{1}{p^\nu} - \sum_{p \le \sqrt{x}} p \sum_{\nu >  \frac{\log x}{\log p}} \frac{1}{p^\nu} 
  = \sum_{p \le \sqrt{x}} \frac{1}{p-1} + O \left( \frac{1}{x} \sum_{p \le \sqrt{x}} p \right) 
  \end{equation*}
\begin{equation*}
   = \sum_{p \le \sqrt{x}} \frac{1}{p} + \sum_p \frac{1}{p(p-1)} - \sum_{p > \sqrt{x}} \frac{1}{p(p-1)} + O \left( \frac{\pi(\sqrt{x})}{\sqrt x} \right) 
   \end{equation*}
\begin{equation*}
  = \log \log \sqrt{x} + M + \sum_p \frac{1}{p(p-1)} + O \left( \frac{1}{\log x} \right) 
  \end{equation*}
\begin{equation*}
  = \log \log \sqrt{x} + F + O \left( \frac{1}{\log x} \right)
\end{equation*}
and
$$
\left| \Sigma_2 \right| \le \log x \sum_{p \leq \sqrt{x}} \frac{p}{\log p} \ll  \sqrt{x} \log x \sum_{p\le \sqrt{x}} \frac1{\log p}  \ll 
\frac{x}{\log x},
$$
by \eqref{est_split}, as required.
\end{proof}

\begin{proof}[Proof of Theorem {\rm \ref{Th_B_gcd}}]
With the help of Lemma \ref{Lemma_additiv_gcd}, similar to the proofs of above, 
\begin{equation*}
   \sum_{n_1,\ldots,n_k \le x} B\left( (n_1,\dotsc,n_k) \right) 
 = \sum_{\substack{p^\nu \le x \\ \nu \ge 2}} p \left \lfloor \frac{x}{p^\nu} \right \rfloor^k 
\end{equation*}
\begin{equation*} 
  = \sum_{p \le \sqrt{x}} p \sum_{2 \le \nu \le \log x/\log p} \left( \left( \frac{x}{p^\nu} \right)^k + O \left( \left( \frac{x}{p^\nu} \right)^{k-1} \right) \right) \\
\end{equation*}
\begin{equation*}
= x^k \sum_{p \le \sqrt{x}} \ \sum_{\nu = 2}^\infty \frac{1}{p^{\nu k-1}} - x^k \sum_{p \le \sqrt{x}} \sum_{\nu > \log x/ \log p} 
   \frac{1}{p^{\nu k-1}} + O \left( x^{k-1} \sum_{p \le \sqrt{x}} \ \sum_{2 \le \nu \le \log x/ \log p} \frac{1}{p^{\nu(k-1)-1}} \right) \\
\end{equation*}
\begin{equation*}
= x^k \sum_p \frac{1}{p^{k-1}(p^k-1)} - x^k \sum_{p > \sqrt{x}} \frac{1}{p^{k-1}(p^k-1)} + O \left( \sum_{p \le \sqrt{x}} \frac{p^{k+1}}{p^k-1} \right) + O \left( x^{k-1} \sum_{p \le \sqrt{x}} \ \sum_{\nu = 2}^\infty \frac{1}{p^{\nu(k-1)-1}} \right) \\
\end{equation*}
\begin{equation*}
= x^k \sum_p \frac{1}{p^{k-1}(p^k-1)} + O \left( \frac{x}{\log x} \right)  + O \left( \sqrt{x} \, \pi(\sqrt{x}) \right) + O \left( x^{k-1} \sum_{p \le \sqrt{x}} \frac{p^{k-3}}{p^k-p} \right) \\
\end{equation*}
\begin{equation*}
= x^k \sum_p \frac{1}{p^{k-1}(p^k-1)} + O \left( \frac{x}{\log x} \right) + O \left( x^{k-1} U_k(x) \right), 
\end{equation*}
where $U_k(x)=1$ ($k\ge 3$) and $U_2(x)= \log \log x$ ($k=2$), completing the proof.
\end{proof}

\begin{proof}[Proof of Theorem {\rm \ref{Th_B_lcm}}]
Similar to the proofs of Theorems \ref{Th_f_additive_lcm} and \ref{Th_f_additive_lcm_1}, by using Proposition \ref{Prop_key} and Theorems \ref{th:B_1}, \ref{Th_B_gcd}.
We omit the details.
\end{proof}

\end{document}